\documentclass[12pt]{amsart}
\usepackage{amsrefs,amssymb,amsmath,amssymb,amsthm,verbatim}
\usepackage[ocgcolorlinks,linktoc=all]{hyperref}
\usepackage{lipsum}
\hypersetup{citecolor=blue,linkcolor=magenta}
\usepackage{cleveref}
\usepackage{color,soul}
\usepackage{mathtools}
\usepackage{esint}

\newtheorem{theorem}{Theorem}[section]

\newtheorem*{thmmain}{Theorem}
\newtheorem{lemma}[theorem]{Lemma}

\newtheorem{qu}{Question}

\theoremstyle{definition}
\newtheorem{definition}[theorem]{Definition}

\theoremstyle{remark}
\numberwithin{equation}{section}

\theoremstyle{remark}
\newtheorem{remark}[theorem]{Remark}

\numberwithin{equation}{section}

\begin{document}

\title[global and local stability results]
 {On the stability of the $L_p$-curvature}
\author[M. N. Ivaki]{Mohammad N. Ivaki}

\begin{abstract}
It is known that the $L_p$-curvature of a smooth, strictly convex body in $\mathbb{R}^{n}$ is constant only for origin-centred balls when $1\neq p>-n$, and only for balls when $p=1$. If $p=-n$, then the $L_{-n}$-curvature is constant only for origin-symmetric ellipsoids. We prove `local' and `global' stability versions of these results. For $p\geq 1$, we prove a global stability result: if the $L_p$-curvature is almost a constant, then the volume symmetric difference of $\tilde{K}$ and a translate of the unit ball $B$ is almost zero. Here $\tilde{K}$ is the dilation of $K$ with the same volume as the unit ball. For $0\leq p<1$, we prove a similar result in the class of origin-symmetric bodies in the $L^2$-distance. In addition, for $-n<p<0$, we prove a local stability result: There is a neighborhood of the unit ball that any smooth, strictly convex body in this neighborhood with `almost' constant $L_p$-curvature is `almost' the unit ball. For $p=-n$, we prove a global stability result in $\mathbb{R}^2$ and a local stability result for $n>2$ in the Banach-Mazur distance.
\end{abstract}
\maketitle
\section{Introduction}
A compact convex subset of $\mathbb{R}^{n}$, $n$-dimensional Euclidean space, with non-empty interior is called a \emph{convex body}. The set of convex bodies in $\mathbb{R}^{n}$ is denoted by $\mathcal{K}^n$ and those with the origin contained in the interior are denoted by $\mathcal{K}^n_0$. We write $\mathcal{F}^n$, and $\mathcal{F}^n_0$ for the set of smooth, strictly convex bodies in $\mathcal{K}^n$ and $\mathcal{K}^n_{0}$ respectively. Moreover, $B$ and $S^{n-1}$ denote the unit ball and the unit sphere of $\mathbb{R}^n$ respectively. $\tilde{K}$ denotes the dilation of $K$ whose $n$-dimensional Lebesgue measure, $V(\tilde{K})$, equals to that of the unit ball; $V(\tilde{K})=V(B)=\kappa_n.$

The support function of a convex body $K$ is defined by
\[h_K(u):= \max_{x\in K} x\cdot u,\quad \forall u\in S^{n-1}.\]

Let $K\in \mathcal{F}^n_0$ and $\nu_K: \partial K\to S^{n-1}$ be the Gauss map which takes $x$ on the boundary of $K$ to its unique outer unit normal vector, and let\[\nu_K^{-1}:S^{n-1}\to\mathbb{R}^{n}\] be the Gauss parameterization of $\partial K$. In this case, we have
 \[h_K(u)= u\cdot\nu_K^{-1}(u).\]
We write $g,\nabla$ for the standard round metric and the corresponding Levi-Civita connection of the unit sphere.
The Gauss curvature of $\partial K$, $\mathcal{K}_K$, and the curvature function of $\partial K$, $f_K$ (as a function on the unit sphere), are related to the support function of the convex body by \[f_K=\frac{1}{\mathcal{K}_K\circ\nu_K^{-1}}=\frac{\det(\nabla^2_{i,j}h_K+g_{ij}h_K)}{\det(g_{ij})}.\]
The function $h_K^{1-p}f_K$ is called the $L_p$-curvature function of $K.$

For $K\in \mathcal{F}_0^n$ we define the scale invariant quantity
\[\mathcal{R}_p(K)=\max_{S^{n-1}} (h_K^{1-p}f_K)/\min_{S^{n-1}} (h_K^{1-p}f_K).\]

We prove stability versions of the following theorem, which is due to a collective work of Firey, Lutwak, Andrews, Brendle, Choi, and Daskalopoulos \cites{Fir74,Lut93,An99a,BCD}: 

\begin{thmmain}
Let $p\in (-n,\infty),\,p\neq1$. If $K\in\mathcal{F}_0^n$ satisfies \[h_K^{1-p}f_K\equiv 1,\] then $K$ is the unit ball.
\end{thmmain}

\begin{qu}\label{stability p geq 1}
Is there an increasing function $f$ with $\lim\limits_{\varepsilon\to 0}f(\varepsilon)=0$ with the following property? If $K\in\mathcal{F}_0^n$ satisfies
$\mathcal{R}_{p}(K)\leq 1+\varepsilon,$
then $K$ is $f(\varepsilon)$-close to a ball in a suitable sense?
\end{qu}

For $p\geq 1$ due to the $L_p$-Minkowski inequality and a refinement of Urysohn's inequality from \cite{Seg12} we can obtain a global stability result in a very strong sense. The relative asymmetry of two convex bodies $K,L$ is defined as
\[\mathcal{A}(K,L):=\inf_{x\in\mathbb{R}^n}\frac{V(K\Delta (\lambda L+x))}{V(K)},\quad \mbox{where}~\lambda^n=\frac{V(K)}{V(L)}\]
and $K\Delta L=(K\setminus L)\cup (L\setminus K).$ 
\begin{theorem}\label{lp>1}
Let $p\geq1$. There exists a constant $C$ independent of dimension with the following property. Any $K\in\mathcal{F}_0^n$ satisfies
\[\mathcal{A}(\tilde{K},B)\leq Cn^{2.5} \left(\mathcal{R}_{p}(K)^{\frac{1}{p}}-1\right)^{\frac{1}{2}}\]
\end{theorem}

Interestingly, the $L_{p+1}$-Minkowski inequality also allows us to prove the global stability for $0\leq p< 1$ in the class of origin-symmetric bodies in the $L^2$-distance. The $L^2$-distance of $K,L$ is defined by
\[\delta_2(K,L)=\left(\frac{1}{\omega_n}\int|h_K-h_L|^2d\sigma\right)^{\frac{1}{2}}.\]
Here $\sigma$ is the spherical Lebesgue measure on $S^{n-1},$ and $\omega_i$ is the surface area of the $i$-dimensional ball.
\begin{theorem}\label{0<lp<1}
Let $0\leq p< 1$ and $K\in\mathcal{F}^n$ be origin-symmetric. There exists an origin-centred ball $B_r$ with radius $1\leq r\leq \mathcal{R}_p(K)$, such that
\[ 
\delta_2(\tilde{K},B_r)\leq D(\tilde{K})\left(1-\mathcal{R}_{p}(K)^{-1}\right)^{\frac{1}{2}}.
\]
Here the diameter of $\tilde{K},$ $D(\tilde{K})$, satisfies the inequality
\[D(\tilde{K})\leq 2\left(\left(1+\left(\frac{4\omega_{n-1}}{\omega_n}\right)^{\frac{1}{2}}\right)\mathcal{R}_p(K)\right)^3.\]
\end{theorem}
For $p\in (-n,0)$, we also establish a local stability result. The points $e_p$ will be defined in \Cref{def functionals}. 

\begin{theorem}\label{lp}
Let $p\in(-n,0).$ There exist positive constants $\gamma,\, \delta$, depending only on $n,\, p$ with the following property. If $K\in\mathcal{F}_0^n$
with $e_p(K)=0$ satisfies $|h_{\lambda K}-1|_{C^{3}}\leq \delta$ for some $\lambda>0$, then
\[\delta_{2}(\tilde{K},B)\leq \gamma \left(\mathcal{R}_{p}(K)-1\right).\]
\end{theorem}

 \begin{remark}
For the case $p=0$, some progress recently has been made on the stability of the cone-volume measure by B{\"{o}}r{\"{o}}czky and De in \cite{BD2021} based on the logarithmic Minkowski inequality in the class of convex bodies with many symmetries proved by B{\"{o}}r{\"{o}}czky and Kalantzopoulos in \cite{BK2021}. Although our proof of \Cref{0<lp<1} is independent of the existence of $L_p$-Minkowski inequality for $0\leq p<1$, it is worth pointing out that such an inequality exists in some particular cases: $p\in [0,1)$ and in the class of origin-symmetric convex bodies in the plane, or in any dimension and in the class of origin-symmetric bodies for $p\in (p_0,1)$ where $p_0>0$ is some constant depending on $n$; see \cites{CYLL20,KM17, Mil21,Mil21b}.
 \end{remark}

Let $K\in \mathcal{F}^n_0$. The centro-affine curvature of $K$, $H_K$, is defined by
\[H_K:=(h_K^{n+1}f_K)^{-1}.\]
It is known that $H_K(u)$ is (up to a constant) just a power of the volume of the origin-centred ellipsoid touching $K$ at $\nu_K^{-1}(u)$ of second-order (osculating ellipsoid), and thus is an $SL(n)$ covariant notion.
In particular, one of the key properties of the centro-affine curvature is that $\min H_K$ and $\max H_K$ are invariant under special linear transformation $SL(n)$. That is,
\begin{equation}\label{eq: sln invariant property}
\min_{S^{n-1}} H_K=\min_{S^{n-1}} H_{\ell K},~\max_{S^{n-1}} H_K=\max_{S^{n-1}} H_{\ell K},\quad \forall\ell\in SL(n).
\end{equation}

A remarkable theorem of Pogorelov states that if the centro-affine curvature of a smooth, strictly convex body is constant, then the body is an origin-centred ellipsoid; cf. \cite{Gut}*{Thm. 10.5.1}, \cites{calabi,CH,MdP14}. It is of great interest to find a stability version of this statement, for example, in the Banach-Mazur distance $d_{\mathcal{BM}}$. For two convex bodies $K,L$, $d_{\mathcal{BM}}(K,L)$ is defined by
\begin{multline}
\min\{ \lambda\geq 1: (K-x)\subseteq \ell(L-y)\subseteq \lambda(K-x),
\ell\in GL(n),~ x,y\in\mathbb{R}^n\}.\nonumber
\end{multline}

\begin{qu}\label{stability q}
Is there an increasing function $f$ with $\lim\limits_{\varepsilon\to 0}f(\varepsilon)=0$ with the following property? If $K\in\mathcal{F}_0^n$ satisfies
\[\mathcal{R}_{-n}(K)=\frac{\max H_K}{\min H_K}\leq 1+\varepsilon,\]
then $K$ is $f(\varepsilon)$-close to an ellipsoid in the Banach-Mazur distance.
\end{qu}

The following theorem gives a positive answer to this question in the plane under no additional assumption.
\begin{theorem} \label{2d-affine}
There exist $\gamma,\,\delta>0$ with the following property. If $K\in\mathcal{F}_0^2$ satisfies $\mathcal{R}_{-2}(K)\leq1+\delta,$
then we have
\[(d_{\mathcal{BM}}(K,B)-1)^4\leq \gamma \left(\mathcal{R}_{-2}(K)-1\right).\]
If $K$ has its Santal\'{o} point at the origin, then
\[(d_{\mathcal{BM}}(K,B)-1)^4\leq \gamma (\sqrt{\mathcal{R}_{-2}(K)}-1).\]
Moreover, if $K$ is origin-symmetric, then
\[d_{\mathcal{BM}}(K,B)\leq \sqrt{\mathcal{R}_{-2}(K)}.\]
In this case, we may allow $\delta=\infty.$
\end{theorem}

In higher dimensions, we have the following `local' stability result.

\begin{theorem}\label{nd-affine}
There exist positive numbers $\gamma,\,\delta$, depending only on $n$ with the following property. Suppose $K\in\mathcal{F}_0^n$
has its Santal\'{o} point at the origin, and for some $\ell\in GL(n)$ we have $|h_{\ell K}-1|_{C^{3}}\leq \delta$. Then
\[d_{\mathcal{BM}}(K,B)\leq \gamma \left(\mathcal{R}_{-n}(K)-1\right)^{\frac{1}{3(n+1)}}+1.\]
\end{theorem}

\section{background}
A convex body $K$ is said to be of class $C^{2}_{+}$, if its boundary hypersurface is two-times continuously differentiable and the support function is differentiable.

Let $K,L$ be two convex bodies with the origin of $\mathbb{R}^{n}$ in their interiors. In the following, we put $a\cdot K:=a^{\frac{1}{p}}K$ and $b\cdot L:=b^{\frac{1}{p}}L$ where $a,b>0$. For $p\geq 1$, the $L_p$-linear combination $a\cdot K+_{p}b\cdot L$ is defined as the convex body whose support function is given by $(ah_K^p+b h_L^p)^{\frac{1}{p}}.$

For $K,L\in \mathcal{K}_0^n,$ the mixed $L_p$-volume $V_{p}(K,L)$ is defined as the first variation of the usual volume with respect to the $L_p$-sum:
\[\frac{n}{p}V_{p}(K,L)=\lim_{\varepsilon\to 0^+}\frac{V(K+_p\varepsilon\cdot L)-V(K)}{\varepsilon}.\]
Aleksandrov, Fenchel and Jessen for $p=1$ and Lutwak \cite{Lut93} for $p>1$ have shown that there exists a unique Borel measure $S_{p}(K,\cdot)$ on $ S^{n-1}$, $L_p$-surface area measure, such that
\[V_p(K,L)=\frac{1}{n}\int h_L^p(u)dS_{p}(K,u).\]
Moreover, $S_{p}(K,\cdot)$ is absolutely continuous with respect to the surface area measure of $K$, $S(K,\cdot)$, and has the Radon--Nikodym derivative
\[\frac{dS_{p}(K,\cdot)}{dS(K,\cdot)}=h_K^{1-p}(\cdot).\]

The measure $dS_{p,K}=h_K^{1-p}dS_K$ is known as the $L_p$-surface area measure. If the boundary of $K$ is $C_+^2$, then
\[\frac{dS_K}{d\sigma}=\frac{1}{\mathcal{K}_K\circ \nu_K^{-1}}=f_K.\]

For $p>1$, the $L_p$-Minkowski inequality states that for convex bodies $K,L$ with the origin in their interiors we have
\begin{align*}
	\frac{1}{n}\int h_L^p dS_p(K)\geq V(K)^{1-\frac{p}{n}}V(L)^{\frac{p}{n}},
\end{align*}
with equality holds if and only if $K$ and $L$ are dilates (i.e. for some $\lambda>0$, $K=\lambda L$); see \cite{Lut93}. For $p=1$, the same inequality holds for all $K,L\in\mathcal{K}^n,$ and equality holds if and only if $K$ is homothetic to $L.$

The polar body, $K^{\ast}$, of $K\in \mathcal{K}^n_0$ is the convex body defined by
\[K^{\ast}=\{y\in\mathbb{R}^n: x\cdot y\leq 1,~\forall x\in K\}.\]
All geometric quantities associated with the polar body are furnished by $\ast.$ For $x\in\operatorname{int}K$, let $K^x:=(K-x)^{\ast}.$ The Santal\'{o} point of $K$, denoted by $s=s(K)$, is the unique point in $\operatorname{int}K$ such that
\[V(K^s)\leq V(K^x)\quad \forall x\in\operatorname{int}K.\]
If $K=-K$, then $s(K)=0$ and $K^{\ast}=K^{s}.$ 

The Blaschke-Santal\'{o} inequality states that
\begin{equation*}
V(K^s)V(K)\leq V(B)^2,
\end{equation*}
and equality holds if and only if $K$ is an ellipsoid. 
\begin{definition}\label{def functionals}
The $L_p$-widths of $K\in \mathcal{K}^n$ are defined as follows.
\begin{enumerate}
\item For $p> 1:$ $\mathcal{E}_p(K)=\frac{1}{\omega_n}\inf_{x\in\operatorname{int}K}\int h_{K-x}^pd\sigma .$
 \item For $p=0:$
$\mathcal{E}_0(K)=\frac{1}{\omega_n}\sup_{x\in \operatorname{int}K}\int \log h_{K-x}d\sigma.$
 \item For $0<p<1:$
$\mathcal{E}_p(K)=\frac{1}{\omega_n}\sup_{x\in\operatorname{int}K}\int h_{K-x}^pd\sigma.$
 \item For $-n\leq p<0:$
$\mathcal{E}_p(K)=\frac{1}{\omega_n}\inf_{x\in\operatorname{int}K}\int h_{K-x}^pd\sigma.$
\end{enumerate}
Here $\omega_n=n\kappa_n=\int d\sigma.$
\end{definition}
Here and in the sequel, $e_p$ denotes the unique point at which the corresponding $\sup$ or $\inf$ is attained. The points $e_p$ are \emph{always} in the interior of the convex body; see e.g. \cite{I2}*{Lem. 3.1}. If $K$ is origin-symmetric, then $e_p(K)$ lies at the origin.

For $p\geq 1$ by the $L_p$-Minkowski inequality we have
\begin{align}\label{ineq p>1}
	\mathcal{E}_p(\tilde{K})\geq 1.
\end{align}
For $p\in (-n,0]$ by the Blaschke-Santal\'{o} inequality,
\begin{align}\label{nonpositive p}
\mathcal{E}_0(\tilde{K})\geq 0,\quad\mathcal{E}_p(\tilde{K})\leq 1,
\end{align}
and equality holds when $K$ is a ball.
Moreover, for $p\in (0,1)$ we have
\begin{align}\label{eq: p to -p}
\mathcal{E}_{p}(\tilde{K})\mathcal{E}_{-p}(\tilde{K})
&=\frac{1}{\omega_n^2}\int h_{\tilde{K}-e_{p}(\tilde{K})}^pd\sigma\int h_{\tilde{K}-e_{-p}(\tilde{K})}^{-p}d\sigma\\
&\geq
\frac{1}{\omega_n^2}\int h_{\tilde{K}-e_{-p}(\tilde{K})}^pd\sigma\int h_{\tilde{K}-e_{-p}(\tilde{K})}^{-p}d\sigma\geq 1,\nonumber
\end{align}
where we used the definition of $e_{p}$ in the last line.
Therefore we obtain
\begin{align}\label{positive p}
\mathcal{E}_{p}(\tilde{K})\geq 1,
\end{align}
and the equality holds only for balls.

We conclude this section by remarking that $\mathcal{E}_p$ enjoys the second {\L}ojasiewicz-Simon gradient inequality; see \cites{Sim83,Sim96} for further details.

\section{Stability of the width functionals}

In this section we prove the stability of the inequalities \eqref{ineq p>1} and \eqref{nonpositive p} ($p\neq 0$).

\begin{lemma}\label{dist of entropy point to O}
Suppose $p\in[-n,0)$. Let $K\in \mathcal{K}^n$ with $V(K)=V(B)$. Then
\[|e_p(K)-s(K)|^2\leq c_0\left(1-\mathcal{E}_p(K)\right)D(K)^{2-p},\]
where $c_0^{-1}:=\frac{p(p-1)}{2\omega_n}\int (u\cdot v )^2d\sigma(u)=\frac{p(p-1)}{2n}$ for any vector $v,$ and $D(K)$ denotes the diameter of $K.$
\end{lemma}
\begin{proof}
We may suppose $e_p(K)\neq s(K).$
Define $v=-\frac{e_p(K)-s(K)}{|e_p(K)-s(K)|}$ and 
\[e(t)=e_p(K)+tv,\quad t\in [0,|e_p(K)-s(K)|].\]
Let us denote the support function of
$K-e(t)$ by $h_t$ and \[E(t):=\frac{1}{\omega_n}\int h_t^pd\sigma.\] 
Note that $E(0)=\mathcal{E}_p(K)$ and $E'(0)=0$. Moreover, the second derivative of $E$ is given by
\begin{align*}
	E''(t)&=\frac{p(p-1)}{\omega_n}\int h_t^{p-2}(u)(u\cdot v)^2d\sigma(u).
\end{align*}
Due to $h_t\leq D(K)$ we obtain
\[D(K)^{p-2}|e_p(K)-s(K)|^2\leq c_0\left(\frac{1}{\omega_n}\int h_{K-s(K)}^pd\sigma-\mathcal{E}_p(K)\right).\]
Now the claim follows from the Blaschke-Santal\'{o} inequality.
\end{proof}

\begin{theorem}\label{thm: stability p} The following statements hold.
\begin{enumerate}

\item Let $p\geq 1.$ If $\mathcal{E}_p(\tilde{K})\leq 1+\varepsilon$, then 
 \[\mathcal{A}(\tilde{K},B)^2\leq Cn^{5}\left((1+\varepsilon)^{\frac{1}{p}}-1\right).\]
Here $C$ is a universal constant that does not depend on $n.$
\item Let $p\in (-n,0).$ If $\mathcal{E}_p(\tilde{K})\geq 1-\varepsilon$, then there exists an origin-centred ball of radius $r$, $B_r$, such that
 	\begin{align*}
\delta_2(\tilde{K}-e_p(\tilde{K}),B_r)\leq \left(2c_1\left(D(\tilde{K})+r\right)^{n+1}\varepsilon\right)^{\frac{1}{2}}+\left(c_0 D(\tilde{K})^{2-p}\varepsilon\right)^{\frac{1}{2}}.
\end{align*}
\end{enumerate}
Moreover, if $\tilde{K}$ is origin-symmetric, then the last term on the right-hand-side can be dropped and $D(\tilde{K})$ can be replaced by $\frac{1}{2}D(\tilde{K})$. Here
\begin{align*}
1\leq r\leq (1-\varepsilon)^{\frac{1}{p}},\quad
c_1:=\max\left\{\frac{n}{p+n},-\frac{n}{p}\right\},
\end{align*}
and $c_0$ is the constant from \Cref{dist of entropy point to O}.
\end{theorem}
\begin{proof}
\emph{Case} $p\geq 1$: Since $\mathcal{E}_p(\tilde{K})\leq 1+\varepsilon,$ we have
\[\frac{1}{\omega_n}\int h_{\tilde{K}}d\sigma\leq \mathcal{E}_p(\tilde{K})^{\frac{1}{p}}\leq (1+\varepsilon)^{\frac{1}{p}}.\]
The refinement of Urysohn's inequality in \cite{Seg12} completes the proof.

\emph{Case} $-n<p<0:$ Assume $V(K)=V(B)$. We denote the support function of $K-e_{p}(K)$ by $h_{p}$ and the support function of $K-s(K)$ by $h_s.$ Since $s(K),\,e_{p}(K)$ are in the interior of $K$, both $h_{s}$ and $h_{p}$ are positive functions.

Let us put
\[f=h_{s}^{p},~ g=1,~ a=-\frac{n}{p},~b=\frac{n}{n+p},~c_1=\max \{a,b\}.\]
By \cite{Aldaz2008stability}*{Thm. 2.2}, we have
\begin{align}\label{aldaz}
\frac{\int h_s^{p}d\sigma}{\left(\int\frac{1}{h_{s}^{n}}d\sigma\right)^{-\frac{p}{n}}
\omega_n^{\frac{p+n}{n}}}\leq 1-\frac{1}{c_1}\left|\frac{h_s^{-\frac{n}{2}}}
{\left(\int\frac{1}{h_{s}^{n}}d\sigma\right)^{\frac{1}{2}}}-
\frac{1}{\omega_n^{\frac{1}{2}}}\right|_{L^2}^2.
\end{align}

Due to our assumption,
\begin{align}\label{eq: mni2}
\int h_{s}^{p}d\sigma\geq \int h_{p}^{p}d\sigma\geq \omega_n(1-\varepsilon).
\end{align}
By the Blaschke-Santal\'{o} inequality, we have
\begin{align}\label{eq: mni}
\int \frac{1}{h_{s}^{n}}d\sigma\leq \omega_n.
\end{align}
From (\ref{eq: mni2}), (\ref{eq: mni}), it follows that
\begin{align}\label{litttle step}
1-\varepsilon &\leq\frac{\int h_s^{p}d\sigma}{\left(\int \frac{1}{h_{s}^{n}}d\sigma\right)^{-\frac{p}{n}}
\omega_n^{\frac{p+n}{n}}},\\
(1-\varepsilon)\omega_n&\leq \int h_{s}^{p}d\sigma\leq \left(\int \frac{1}{h_{s}^{n}}d\sigma\right)^{-\frac{p}{n}}\omega_n^{\frac{n+p}{n}}.\nonumber
\end{align}
Combining \eqref{aldaz} and \eqref{litttle step} we obtain
\begin{align}\label{xxyyzz}
\left|h_{s}^\frac{n}{2}-r^\frac{n}{2}\right|_{L^2}^2\leq c_1\omega_n D(K)^{n}\varepsilon,
\end{align}
where
\begin{align}\label{ie: another gamma determine step 1}
r^n:=\omega_n\left(\int \frac{1}{h_{s}^{n}}d\sigma\right)^{-1},\quad 1\leq r\leq (1-\varepsilon)^{\frac{1}{p}}.
\end{align}
In view of \eqref{xxyyzz} and \eqref{ie: another gamma determine step 1} we have
 \begin{align}\label{ie: even n}
 \left|h_{s}-r\right|_{L^2}^2\leq c_1\omega_n (D(K)^{\frac{1}{2}}+r^{\frac{1}{2}})^2D(K)^{n}\varepsilon.
\end{align}

If $K$ is origin-symmetric, then $s(K)=e_p(K)$ and the proof is complete. Moreover, in this case we could have replaced $D(K)$ by $\frac{1}{2}D(K)$. Otherwise, to bound $\left|h_p-r\right|_{L^2}$, note that by \Cref{dist of entropy point to O} we have
\[|e_p(K)-s(K)|^2\leq c_0 D(K)^{2-p}\varepsilon.\]
Therefore,
\begin{align*}
\left|h_p-r\right|_{L^2}\leq{}& \left|h_s-r\right|_{L^2}+\omega_n^{\frac{1}{2}}|e_p(K)-s(K)|\\
\leq{}& \left(2c_1\omega_n\left(D(K)+r\right)^{n+1}\varepsilon\right)^{\frac{1}{2}}+\left(c_0\omega_n D(K)^{2-p}\varepsilon\right)^{\frac{1}{2}}.
\end{align*}
\end{proof}
\begin{remark}
The exponent $1/2$ in (1) is sharp; cf. \cite{FMP10}. Moreover, using \cite{Sch}*{Thm. 7.2.2} it is also possible to give a stability result of order $1/(n+1)$ in (1) for the Hausdorff distance $d_{\mathcal{H}}(\tilde{K}-\operatorname{cent}(\tilde{K}),B)$; we leave out the details to the interested reader. By cutting off opposite caps of height $\varepsilon$ of the unit ball, one can see that the optimal order cannot be better than $1$ in (2). 
\end{remark}

For proving \Cref{0<lp<1}, we only need the stability of the $L_p$-width functional for $p=-1$. In this case, we give a slightly better stability result along with a diameter bound.
\begin{theorem}\label{new stability diam bound p-1}
	Suppose $K$ is an origin-symmetric convex body with
	\[\mathcal{E}_{-1}(\tilde{K})\geq 1-\varepsilon\quad \mbox{for some}~\varepsilon\in (0,1).\]
Then there exists an origin-centred ball $B_r$ of radius $1\leq r\leq (1-\varepsilon)^{-1}$ such that
\begin{align*}
	\delta_2(\tilde{K},B_r)&\leq D(\tilde{K})\sqrt{\varepsilon}.
\end{align*}
Moreover, we have
\[\left(\frac{1}{2}D(\tilde{K})\right)^{\frac{1}{3}}\leq \left(1+\left(\frac{4\omega_{n-1}}{\omega_n}\right)^{\frac{1}{2}}\right)\frac{1}{1-\varepsilon}.\]
\end{theorem}

\begin{proof}
Set $h=h_{\tilde{K}}$. We have
\begin{align*}
\frac{\int \frac{1}{h}d\sigma}{\left(\int \frac{1}{h^2}d\sigma\right)^{\frac{1}{2}}
\omega_n^{\frac{1}{2}}}= 1-\frac{1}{2}\left|\frac{\frac{1}{h}}
{\left(\int \frac{1}{h^2}d\sigma\right)^{\frac{1}{2}}}-
\frac{1}{\omega_n^{\frac{1}{2}}}\right|_{L^2}^2.
\end{align*}
By our assumption and the Blaschke-Santal\'{o} inequality,
\begin{align*}
\int \frac{1}{h}d\sigma\geq \omega_n(1-\varepsilon),\quad \int \frac{1}{h^2}d\sigma\leq \omega_n.
\end{align*}
Therefore,
\begin{align*}
1-\varepsilon \leq\frac{\int \frac{1}{h}d\sigma}{\left(\int \frac{1}{h^2}d\sigma\right)^{\frac{1}{2}}
\omega_n^{\frac{1}{2}}},\quad 
(1-\varepsilon)\omega_n\leq \int \frac{1}{h}d\sigma\leq \left(\int \frac{1}{h^2}d\sigma\right)^{\frac{1}{2}}\omega_n^{\frac{1}{2}}.\nonumber
\end{align*}
Combining these inequalities we obtain
\begin{align*}
	\left|h-r\right|_{L^2}^2\leq \omega_n D(\tilde{K})^2\varepsilon,
\end{align*}
where $r^2:=\omega_n\left(\int \frac{1}{h^2}d\sigma\right)^{-1}$ and
$
1\leq r\leq (1-\varepsilon)^{-1}.
$

Next we estimate the diameter from above. Define
	\[S=\{v\in S^{n-1}: h_{\tilde{K}}(v)\leq R^{\frac{1}{3}}\},\]
	where $R:=\max h_{\tilde{K}}=h_{\tilde{K}}(u)$ for some vector $u\in S^{n-1}$. We may assume $R>1.$
Then by the Blaschke-Santal\'{o} inequality we have
\begin{align*}
	(1-\varepsilon)\omega_n&\leq \int_{S} \frac{1}{h_{\tilde{K}}}d\sigma+\int_{S^c}\frac{1}{h_{\tilde{K}}}d\sigma\\
	&\leq \left(\int_{S} \frac{1}{h_{\tilde{K}}^2}d\sigma\right)^{\frac{1}{2}}|S|^{\frac{1}{2}}+\frac{|S^c|}{R^{\frac{1}{3}}}\\
	&\leq (\omega_n)^{\frac{1}{2}}|S|^{\frac{1}{2}}+\frac{\omega_n}{R^{\frac{1}{3}}}.
\end{align*}
Moreover, by convexity we have $h_{\tilde{K}}(v)\geq R|u\cdot v|$ for all $v\in S^{n-1}.$
Hence if $v\in S$, then $|u\cdot v|\leq R^{-\frac{2}{3}}.$ Now using 
\[\frac{\pi}{2}-\arccos x\leq 2x,\quad \forall x\in [0,1],\] we obtain
\[\frac{1}{2}|S|\leq\omega_{n-1}\int_{\arccos R^{-\frac{2}{3}}}^{\frac{\pi}{2}}\sin^{n-2}\theta d\theta\leq \frac{2\omega_{n-1}}{R^{\frac{2}{3}}},\]
Therefore,
\[1-\varepsilon\leq \left(1+\left(\frac{4\omega_{n-1}}{\omega_n}\right)^{\frac{1}{2}}\right)\frac{1}{R^{\frac{1}{3}}}.\]
\end{proof}
We are ready to give the proofs of our main theorems.
\begin{proof}[Proof of \Cref{lp>1}]
Suppose
$m \leq h_K^{1-p}dS_K/d\sigma\leq M.$
Therefore by the $L_p$-Minkowski inequality,
\begin{align*}
	\frac{m}{n}\frac{\int h_K^pd\sigma}{V(B)^{1-\frac{p}{n}}V(K)^{\frac{p}{n}}} &\leq \frac{1}{n}\frac{\int h_K^ph_K^{1-p}dS_K}{V(B)^{1-\frac{p}{n}}V(K)^{\frac{p}{n}}}	 \\
&= \frac{V(K)^{1-\frac{p}{n}}}{V(B)^{1-\frac{p}{n}}}\nonumber\\
&\leq \frac{V(B)^{-\frac{p}{n}}\frac{1}{n}\int h_K^{1-p}dS_K }{V(B)^{1-\frac{p}{n}}}\leq M.\nonumber
\end{align*}
Hence $\mathcal{E}_p(\tilde{K})\leq \mathcal{R}_p(\tilde{K})$,
and by \Cref{thm: stability p} the proof is complete.
\end{proof}
\begin{proof}[Proof of \Cref{0<lp<1}]
Assume
$m \leq h_K^{1-p}dS_K/d\sigma\leq M.$
	Then by the $L_q$-Minkowski inequality for $q=p+1$ we have
\[\frac{1}{n}\int \frac{1}{h_K^{q-p}}h_K^{1-p}dS_K\geq V(K)^{1-\frac{q}{n}}V(B)^{\frac{q}{n}}.\]
Therefore,
\begin{align}\label{new case}
\frac{M}{n}V(K)^{\frac{q-p}{n}}\int \frac{1}{h_K^{q-p}}d\sigma\geq V(K)^{1-\frac{p}{n}}V(B)^{\frac{q}{n}}.	
\end{align}
Owing to \eqref{positive p} for $p>0$ we have
\[V(K)\geq \frac{m}{n}\int h_K ^pd\sigma\geq mV(K)^{\frac{p}{n}}V(B)^{1-\frac{p}{n}}\]
and hence for $p\geq 0$,
\begin{align}\label{lower vol bound}
V(K)^{1-\frac{p}{n}}\geq mV(B)^{1-\frac{p}{n}}.
\end{align}
Since $e_{p-q}(K)=0$, in view of \eqref{new case} we obtain
$\mathcal{E}_{p-q}(\tilde{K})\geq \mathcal{R}_p(K)^{-1}.$
The claim follows from \Cref{new stability diam bound p-1}.
\end{proof}
\begin{remark}\label{slight update}
It is clear from the proofs of \Cref{lp>1} and \Cref{0<lp<1} that if $K$ has only a positive continuous curvature function, then the same conclusions hold.
\end{remark}
\begin{remark}
Applying the Blaschke-Santal\'{o} inequality to the left-hand side of \eqref{new case}, we obtain
\[\left(\frac{V(K)}{V(B)}\right)^{\frac{n-p}{n}}\leq M.\]
This combined with \eqref{lower vol bound} yields
\[m\leq\left(\frac{V(K)}{V(B)}\right)^{\frac{n-p}{n}}\leq M.\]
Hence in the class of origin-symmetric bodies if $V(K)=V(B)$, then for any $p\geq 0$ the $L_p$-curvature function attains the value $1$ at some point; see also \Cref{qu2}. 
\end{remark}
\begin{proof}[Proof of \Cref{lp}]~
	Define $\tilde{\mathcal{E}}_p:\mathcal{F}_0^n\to (0,\infty)$ by
\[\tilde{\mathcal{E}}_p(h_L)= \left(\int h_L^pd\sigma\right)^{\frac{n}{p}}/V(L).\]
By the divergence theorem we have
\[(\operatorname{grad}\tilde{\mathcal{E}}_p)(h_K)=\frac{h_K^{p-1}\left(\int h_K^pd\sigma\right)^{\frac{n}{p}}}{V(K)^2}\left(\frac{nV(K)}{\int h_K^pd\sigma}-h_K^{1-p}f_K\right).\]
By \cite{Sim96}*{Sec. 3.13 (ii)} and \cite{Sim96}*{p. 80}, there exist $c_2,\delta>0$, such that for any $K$ with $|h_K-1|_{C^{3}}\leq \delta$, there holds
\begin{align*}
\left|\tilde{\mathcal{E}}_p(K)-\tilde{\mathcal{E}}_p(B)\right|^{\frac{1}{2}}\leq c_2\left|(\operatorname{grad}\tilde{\mathcal{E}}_p)
(h_K)\right|_{L^2}.
\end{align*}
Assuming $m\leq h_K^{1-p}f_K\leq M$ gives
\[m\leq \frac{nV(K)}{\int h_K^pd\sigma}\leq M. \]
This in turn implies
$\left|\mathcal{E}_p(\tilde{K})^{\frac{n}{p}}-1\right|\leq c_3 \left(\mathcal{R}_p(\tilde{K})-1\right)^{2},$
as well as
\[\mathcal{E}_p(\tilde{K})\geq \left(1+c_3 \left(\mathcal{R}_p(\tilde{K})-1\right)^{2}\right)^{\frac{p}{n}}.\]
Due to \Cref{thm: stability p}, the proof is complete.
\end{proof}

\begin{proof}[Proof of \Cref{2d-affine}]Suppose $m\leq H_K \leq M.$
By \cite{An1}*{Lem. 10},
\begin{align}\label{eq: something}
V(K)\geq \frac{\pi}{\sqrt{M}}.
\end{align}
In fact, the lemma states that if $V(K)=\pi$, then centro-affine curvature at some point attains $1$. Therefore, since $V(\sqrt{\pi/V(K)}K)=\pi$, the function $\left(V(K)/\pi\right)^2H_K$ takes the value $1$ at some point.
Hence using (\ref{eq: something}) and the H\"{o}lder inequality we obtain
\begin{align}\label{id: useful}
V(K)V(K^s)&\ge\frac{\left(\int h_Kf_KH_K^{\frac{1}{3}}
d\sigma\right)^3}{4\int h_Kf_Kd\sigma}\geq mV(K)^2\ge \pi^2\frac{m}{M}.\nonumber
\end{align}

If the Santal\'{o} point is at the origin, then we can obtain a slightly better lower bound for the volume product. By \cites{Hug}, we have
\[
H_K(u)H_{K^{\ast}}(u^{\ast})=1,
\]
where $u$ and $u^{\ast}$ are related by $\langle \nu_K^{-1}(u), \nu_{K^{\ast}}^{-1}(u^{\ast})\rangle=1.$
Since $K^s=K^{\ast}$, this yields
\[\frac{1}{M}\leq H_{K^s}\leq \frac{1}{m},\quad V(K^s)\geq \pi\sqrt{m}.\]
Therefore, 
$
V(K)V(K^{s})\geq \pi^2\sqrt{\frac{m}{M}}.
$
Now in both cases, the result follows from \cite{I}. The third claim is exactly \cite{I}*{Cor. 4}. 
\end{proof}

\begin{qu}\label{qu2}
Given the previous argument, we would like to raise a question. Suppose $K\in \mathcal{F}_0^n$, $n\geq 3$, and $V(K)=V(B).$ Is it true that the centro-affine curvature of $K$ attains the value $1$ at some point?
\end{qu}

\begin{proof}[Proof of \Cref{nd-affine}]
For all $\ell\in GL(n),$ we have
\[s(\ell K)=\ell s(K)=0,\quad d_{\mathcal{BM}}(\ell K,B)=d_{\mathcal{BM}}(K,B).\] Thus we may assume without loss of generality that \[|h_{K}-1|_{C^{3}}\leq \delta,\]
for some $\delta>0$ to be determined.

Define the functional $\mathcal{P}:\mathcal{F}_0^n\to (0,\infty)$ by \[\mathcal{P}(L)=\mathcal{P}(h_L)=\frac{1}{V(L)V(L^{\ast})}.\]
We have
\begin{align}\label{eq: 1}
(\operatorname{grad}\mathcal{P})(h_K)&=\mathcal{P}^2(K)\left(\frac{V(K)}{h_K^{n+1}}-V(K^{\ast})f_K\right)\\
&=\frac{V(K^{\ast})\mathcal{P}^2(K)}{h_K^{n+1}}\left(\frac{V(K)}{V(K^{\ast})}-\frac{1}{H_K}\right).\nonumber
\end{align}

By \cite{Sim96}*{Sec. 3.13 (ii)}, there exist $\delta,c_2>0$ and $\alpha\in (0,1/2]$, such that for any $K$ with
$|h_K-1|_{C^{3}}\leq \delta$, we have
\begin{equation}\label{eq: 2}
\left|\frac{1}{V(K)V(K^{\ast})}-\frac{1}{V(B)^2}\right|^{1-\alpha}\leq c_2\left|(\operatorname{grad}\mathcal{P})(h_K)\right|_{L^2}.
\end{equation}
By \cite{Sim96}*{p. 80} and \cite{Iva18}*{Lem. 4.1, 4.2} we can choose $\alpha=1/2.$

We estimate the right-hand side of (\ref{eq: 2}) in terms of $\mathcal{R}_{-n}(K)-1.$ Note that $m\leq H_K\leq M$ implies that
\[\frac{1}{M}\le\frac{V(K)}{V(K^{\ast})}=\frac{\int h_Kf_Kd\sigma}
{\int h_Kf_KH_Kd\sigma}
\le\frac{1}{m}.\]
Therefore we obtain
\begin{equation}\label{eq: 3}
\frac{1}{M}\leq \frac{V(K)}{V(K^{\ast})}\leq \frac{1}{m}\quad \mbox{and}\quad \left|\frac{V(K)}{V(K^{\ast})}-\frac{1}{H_K}\right|\leq \frac{M-m}{Mm}.
\end{equation}
On the other hand, by \eqref{eq: 3} and the Blaschke-Santal\'{o} inequality,
\begin{equation}\label{eq: 4}
V(K^{\ast})^2\leq MV(B)^2.
\end{equation}
Putting (\ref{eq: 1}), (\ref{eq: 2}), (\ref{eq: 3}), and (\ref{eq: 4}) all together we arrive at
\[\left|\frac{1}{V(K)V(K^{\ast})}-\frac{1}{V(B)^2}\right|^{\frac{1}{2}}\leq c_3
\left(\mathcal{R}_{-n}(K)-1\right)\frac{\mathcal{P}^2(K)\left| h_K^{-n-1}\right|_{L^2}}{V(K^{\ast})}.\]
Since we are in a small neighborhood of the unit ball, the term 
\[\frac{\mathcal{P}^2(K)\left| h_K^{-n-1}\right|_{L^2}}{V(K^{\ast})}\]
is bounded. Using again the Blaschke-Santal\'{o} inequality we obtain
\[1-c_4\left(\mathcal{R}_{-n}(K)-1\right)^2\leq \frac{V(K)V(K^{\ast})}{V(B)^2}.\]
In view of \cite{B}*{Thm. 1.1}, the proof is complete. 
\end{proof}

\section*{Acknowledgment}
I would also like to thank the referee and the editor for the constructive suggestions that have led to a much better presentation of the original manuscript.
\bibliographystyle{amsalpha-nobysame}

\vspace{10mm}

	\textsc{Institut f\"{u}r Diskrete Mathematik und Geometrie,\\ Technische Universit\"{a}t Wien,
		Wiedner Hauptstr 8-10,\\ 1040 Wien, Austria, }\email{\href{mailto:mohammad.ivaki@tuwien.ac.at}{mohammad.ivaki@tuwien.ac.at}}
\end{document}